\documentclass[a4paper,english]{article}
\usepackage[latin9]{inputenc}
\usepackage{mathtools}
\usepackage{amsmath}
\usepackage{amsthm}
\usepackage{amssymb}
\usepackage{mathdots}
\usepackage{undertilde}
\usepackage{setspace}
\usepackage{esint}
\makeatletter

  \theoremstyle{plain}
  \newtheorem{lem}{\protect\lemmaname}[section]
  \theoremstyle{plain}
  \newtheorem{prop}{\protect\propositionname}[section]
  \theoremstyle{definition}
  \newtheorem{defn}{\protect\definitionname}[section]
  \theoremstyle{plain}
  \newtheorem{thm}{\protect\theoremname}[section]
  \theoremstyle{definition}
  \newtheorem{example}{\protect\examplename}[section]
  \theoremstyle{remark}
  \newtheorem*{rem*}{\protect\remarkname}

\usepackage{amsmath}
\usepackage{amsfonts}
\usepackage{amssymb}
\usepackage{amsthm}
\usepackage{bbm}
\usepackage{graphicx}
\usepackage{float}
\usepackage{csvtools}
\usepackage{array}
\usepackage{enumitem}
\usepackage{mathtools}
\usepackage{amssymb}
\usepackage{stmaryrd}
\usepackage[top=1in,bottom=1in,left=1in,right=1in]{geometry}

\AtBeginDocument{

}

\makeatother

\usepackage{babel}
  \providecommand{\definitionname}{Definition}
  \providecommand{\examplename}{Example}
  \providecommand{\lemmaname}{Lemma}
  \providecommand{\propositionname}{Proposition}
  \providecommand{\remarkname}{Remark}
\providecommand{\theoremname}{Theorem}

\begin{document}

\title{On the maximum of a type of random processes}

\author{Xuan Liu\thanks{Mathematical Institute, University of Oxford, Oxford, OX2 6GG, United
Kingdom. Email: xuan.liu@maths.ox.ac.uk}}

\date{March 11, 2016}
\maketitle
\begin{abstract}
Let $\{\Omega,\{\mathcal{F}_{t}\},\mathcal{F},\mathbb{P}\}$ be a
filtered probability space satisfying the usual conditions. We consider
random processes $X_{t}$, $t\in[0,T]$, which satisfy the following
condition:
\begin{equation}
\mathbb{E}\left(\left|\mathbb{E}\left(X_{t}\big|\mathcal{F}_{s}\right)-X_{s}\right|^{p}\right)\le A_{p,h}|t-s|^{ph},\quad\text{for all }0\le s<t\le T.\label{Eqn-17}
\end{equation}
where $p>1$ and $h\in(0,1]$ are some constants satisfying $ph>1$,
and $A_{p,h}$ is a constant depending only on $p$ and $h$. Typical
examples of such processes are martingales and processes with the
following increment control: 

\begin{equation}
\mathbb{E}\left(\left|X_{t}-X_{s}\right|^{p}\right)\le A_{p,h}|t-s|^{ph},\quad\text{for all }s,t\in[0,T],\label{Eqn}
\end{equation}
 We are interested in estimate of the tail probability of the supremum
\begin{equation}
\mathbb{P}\left(\sup_{t\in[0,T]}\left|X_{t}\right|\ge\lambda\right),\label{Eqn-1}
\end{equation}
for which we will show that a Doob type inequality (see Theorem \ref{Prop-1})
holds for processes satisfying (\ref{Eqn-17}). As an application,
we show that with the condition (\ref{Eqn}) given, the decay of (\ref{Eqn-1})
behaves (roughly speaking) in the same manner as the marginal
\[
\mathbb{P}\left(|X_{t}|\ge\lambda\right).
\]

\end{abstract}

\section{A Doob type maximal inequality}
\begin{lem}
\label{Lemma-1}For any $s_{0}$, $t_{0}\in[0,T]$, $s_{0}<t_{0}$,
it holds that 
\[
\mathbb{E}\left(\sup_{s_{0}\le s<t\le t_{0}}\left|\mathbb{E}(X_{t}\big|\mathcal{F}_{s})-X_{s}\right|^{p}\right)\le C_{p,h,\theta}A_{p,h}|t_{0}-s_{0}|^{ph},
\]
where 
\[
C_{p,h,\theta}=[2\zeta(\theta)]^{p-1}\left(\frac{p}{p-1}\right)^{p}\left(\frac{4}{ph-1}\right)^{\theta(p-1)+1}\Gamma\left[\theta(p-1)+1\right],
\]
 with an arbitrary constant $\theta>1$, $\zeta(\theta)=\sum_{m=1}^{\infty}m^{-\theta}$
is the Riemann zeta function and $\Gamma(z)$ is the Gamma function.\end{lem}
\begin{proof}
Let $s,t\in[s_{0},t_{0}]$, $s<t$ be fixed temporarily. Denote by
\[
I_{l}^{m}=[t_{l-1}^{m},t_{l}^{m}]=s_{0}+(t_{0}-s_{0})\times\left[\frac{l-1}{2^{m}},\frac{l}{2^{m}}\right]
\]
 the dyadic sub-intervals of $[s,t]$. Then there exists a sequence
$\{J_{k}\}\subseteq\{I_{l}^{m}:1\le l\le2^{m},m\ge0\}$ such that 

i) $J_{k}$, $k=1,2,\cdots$, are mutually disjoint;

ii) for any $m\ge1$, there are at most two elements of $\{J_{k}\}$
with length $(t_{0}-s_{0})2^{-m}$;

iii) $[s,t]=\cup_{k=1}^{\infty}J_{k}$.\\
Denote $J_{k}=[u_{k-1},u_{k}]$. Then
\begin{align*}
|\mathbb{E}(X_{t}\big|\mathcal{F}_{s})-X_{s}| & =\left|\sum_{k=1}^{\infty}\mathbb{E}\left(\Delta X_{J_{k}}\big|\mathcal{F}_{s}\right)\right|\\
 & \le\sum_{k=1}^{\infty}\mathbb{E}\left[\left|\mathbb{E}\left(\Delta X_{J_{k}}\big|\mathcal{F}_{u_{k-1}}\right)\right|\Big|\mathcal{F}_{s}\right]\\
 & =\sum_{m=1}^{\infty}\sum_{\{J_{k}:|J_{k}|=(t_{0}-s_{0})2^{-m}\}}\mathbb{E}\left[\left|\mathbb{E}\left(\Delta X_{J_{k}}\big|\mathcal{F}_{u_{k-1}}\right)\right|\Big|\mathcal{F}_{s}\right],
\end{align*}
where $\Delta X_{J_{k}}=X_{u_{k}}-X_{u_{k-1}}$. Let $\xi_{l}^{m}=\mathbb{E}\left(\Delta X_{I_{l}^{m}}\big|\mathcal{F}_{t_{l-1}^{m}}\right)$,
$1\le l\le2^{m}$, $m=1,2,\cdots$. For any $\theta>1$, by Jensen's
inequality,
\begin{align*}
|\mathbb{E}(X_{t}\big|\mathcal{F}_{s})-X_{s}|^{p} & \le\left(\sum_{m=1}^{\infty}\frac{1}{\zeta(\theta)m^{\theta}}\cdot\zeta(\theta)m^{\theta}\sum_{\{J_{k}:|J_{k}|=(t_{0}-s_{0})2^{-m}\}}\mathbb{E}\left[\left|\mathbb{E}\left(\Delta X_{J_{k}}\big|\mathcal{F}_{u_{k-1}}\right)\right|\Big|\mathcal{F}_{s}\right]\right)^{p}\\
 & \le\sum_{m=0}^{\infty}\frac{1}{\zeta(\theta)m^{\theta}}\left(\zeta(\theta)m^{\theta}\sum_{\{J_{k}:|J_{k}|=(t_{0}-s_{0})2^{-m}\}}\mathbb{E}\left[\left|\mathbb{E}\left(\Delta X_{J_{k}}\big|\mathcal{F}_{u_{k-1}}\right)\right|\Big|\mathcal{F}_{s}\right]\right)^{p}\\
 & =\zeta(\theta)^{p-1}\sum_{m=0}^{\infty}m^{\theta(p-1)}\left(\sum_{\{J_{k}:|J_{k}|=(t_{0}-s_{0})2^{-m}\}}\mathbb{E}\left[\left|\mathbb{E}\left(\Delta X_{J_{k}}\big|\mathcal{F}_{u_{k-1}}\right)\right|\Big|\mathcal{F}_{s}\right]\right)^{p}\\
 & \le[2\zeta(\theta)]^{p-1}\sum_{m=0}^{\infty}m^{\theta(p-1)}\sum_{\{J_{k}:|J_{k}|=(t_{0}-s_{0})2^{-m}\}}\Bigg(\mathbb{E}\left[\left|\mathbb{E}\left(\Delta X_{J_{k}}\big|\mathcal{F}_{u_{k-1}}\right)\right|\Big|\mathcal{F}_{s}\right]\Bigg)^{p}\\
 & \le[2\zeta(\theta)]^{p-1}\sum_{m=0}^{\infty}m^{\theta(p-1)}\sum_{l=1}^{2^{m}}\sup_{r\in[s_{0},t_{0}]}\left[\mathbb{E}\left(\left|\xi_{l}^{m}\right|\big|\mathcal{F}_{r}\right)\right]^{p},
\end{align*}
 where the inequality in the fourth line is due to the property ii)
of $\{J_{k}\}$. Hence,
\[
\sup_{s_{0}\le s<t\le t_{0}}|\mathbb{E}(X_{t}\big|\mathcal{F}_{s})-X_{s}|^{p}\le[2\zeta(\theta)]^{p-1}\sum_{m=0}^{\infty}m^{\theta(p-1)}\sum_{l=1}^{2^{m}}\sup_{r\in[s_{0},t_{0}]}\left[\mathbb{E}\left(\left|\xi_{l}^{m}\right|\big|\mathcal{F}_{r}\right)\right]^{p}.
\]
By Doob's maximal inequality for martingales, 
\begin{align*}
\mathbb{E}\left(\sup_{s_{0}\le s<t\le t_{0}}|\mathbb{E}(X_{t}\big|\mathcal{F}_{s})-X_{s}|^{p}\right) & \le[2\zeta(\theta)]^{p-1}\sum_{m=1}^{\infty}m^{\theta(p-1)}\sum_{l=1}^{2^{m}}\mathbb{E}\left(\sup_{r\in[s_{0},t_{0}]}\left[\mathbb{E}\left(\left|\xi_{l}^{m}\right|\big|\mathcal{F}_{r}\right)\right]^{p}\right)\\
 & \le[2\zeta(\theta)]^{p-1}\sum_{m=1}^{\infty}m^{\theta(p-1)}\sum_{l=1}^{2^{m}}\left(\dfrac{p}{p-1}\right)^{p}\mathbb{E}\left(\left[\mathbb{E}\left(\left|\xi_{l}^{m}\right|\big|\mathcal{F}_{t_{0}}\right)\right]^{p}\right)\\
 & \le[2\zeta(\theta)]^{p-1}\left(\dfrac{p}{p-1}\right)^{p}\sum_{m=1}^{\infty}m^{\theta(p-1)}\sum_{l=1}^{2^{m}}\mathbb{E}\left(\left|\xi_{l}^{m}\right|^{p}\right)\\
 & \le A_{p,h}[2\zeta(\theta)]^{p-1}\left(\dfrac{p}{p-1}\right)^{p}\sum_{m=1}^{\infty}m^{\theta(p-1)}\cdot2^{m}\cdot\left(\frac{|t_{0}-s_{0}|}{2^{m}}\right)^{ph}\\
 & =C_{p,h,\theta}A_{p,h}|t_{0}-s_{0}|^{ph},
\end{align*}
where $C_{p,h,\theta}=[2\zeta(\theta)]^{p-1}\left(\frac{p}{p-1}\right)^{p}\left[\sum_{m=1}^{\infty}m^{\theta(p-1)}\cdot2^{-m(ph-1)}\right]$.
Note that
\begin{align*}
\sum_{m=1}^{\infty}m^{\theta(p-1)}\cdot2^{-m(ph-1)} & \le\sum_{m=1}^{\infty}2^{\theta(p-1)}\int_{m-1}^{m}r^{\theta(p-1)}e^{-r(ph-1)\log2}dr\\
 & \le\left(\frac{4}{ph-1}\right)^{\theta(p-1)+1}\int_{0}^{\infty}r^{\theta(p-1)}e^{-r}dr\\
 & =\left(\frac{4}{ph-1}\right)^{\theta(p-1)+1}\Gamma\left[\theta(p-1)+1\right].
\end{align*}
This completes the proof.
\end{proof}

As an application of Lemma \ref{Lemma-1}, we show that a Doob-type
inequality holds for processes satisfying the condition (\ref{Eqn-17}).
To this end, we shall need the following elementary result.
\begin{lem}
\label{Lemma-2}Let $Y_{t}$, $t\in[0,T]$, be any right continuous
random process such that $Y_{t}$ is integrable for each $t$, and
let $0\le s_{0}<t_{0}\le T$. Then

1) For any stopping time $\tau$with $s_{0}\le\tau\le t_{0}$, it
holds that
\begin{equation}
\left|\mathbb{E}\left(Y_{t_{0}}\big|\mathcal{F}_{\tau}\right)-Y_{\tau}\right|\le\mathbb{E}\left[\sup_{u\in[s_{0},t_{0}]}\left|\mathbb{E}\left(Y_{t_{0}}\big|\mathcal{F}_{u}\right)-Y_{u}\right|\Big|\mathcal{F}_{\tau}\right].\label{Eqn-22}
\end{equation}

2) For any $\lambda>0$, it holds that
\begin{equation}
\mathbb{P}\left(\sup_{u\in[s_{0},t_{0}]}Y_{u}\ge\lambda\right)\le\frac{1}{\lambda}\int_{\left\{ \sup_{u\in[s_{0},t_{0}]}Y_{u}\ge\lambda\right\} }\left[\sup_{u\in[s_{0},t_{0}]}\left|\mathbb{E}\left(Y_{t_{0}}\big|\mathcal{F}_{u}\right)-Y_{u}\right|+Y_{t_{0}}\right]d\mathbb{P}.\label{Eqn-24}
\end{equation}
\end{lem}
\begin{proof}
1) By the right continuity of $Y_{t}$, we may assume that $\tau$
takes only countably many values $\{u_{k}:k=1,2,\cdots\}\subseteq[s_{0},t_{0}]$.
Then
\begin{align*}
\left|\mathbb{E}\left(Y_{t_{0}}\big|\mathcal{F}_{\tau}\right)-Y_{\tau}\right| & =\sum_{k=1}^{\infty}\left|\mathbb{E}\left(Y_{t_{0}}\big|\mathcal{F}_{\tau}\right)-Y_{\tau}\right|1_{\{\tau=u_{k}\}}\\
 & =\sum_{k=1}^{\infty}\left|\mathbb{E}\Big[\left(Y_{t_{0}}-Y_{\tau}\right)1_{\{\tau=u_{k}\}}\Big|\sigma\big(\mathcal{F}_{\tau}\cap\{\tau=u_{k}\}\big)\Big]\right|\\
 & =\sum_{k=1}^{\infty}\left|\mathbb{E}\Big[\mathbb{E}\left(\left(Y_{t_{0}}-Y_{u_{k}}\right)\Big|\mathcal{F}_{u_{k}}\right)1_{\{\tau=u_{k}\}}\Big|\sigma\big(\mathcal{F}_{\tau}\cap\{\tau=u_{k}\}\big)\Big]\right|\\
 & \le\sum_{k=1}^{\infty}\mathbb{E}\left[\left(\sup_{u\in[s_{0},t_{0}]}\left|\mathbb{E}\left(Y_{t_{0}}\big|\mathcal{F}_{u}\right)-Y_{u}\right|\right)1_{\{\tau=u_{k}\}}\Big|\sigma\big(\mathcal{F}_{\tau}\cap\{\tau=u_{k}\}\big)\right]\\
 & =\sum_{k=1}^{\infty}\mathbb{E}\left[\sup_{u\in[s_{0},t_{0}]}\left|\mathbb{E}\left(Y_{t_{0}}\big|\mathcal{F}_{u}\right)-Y_{u}\right|\Big|\mathcal{F}_{\tau}\right]1_{\{\tau=u_{k}\}}\\
 & =\mathbb{E}\left[\sup_{u\in[s_{0},t_{0}]}\left|\mathbb{E}\left(Y_{t_{0}}\big|\mathcal{F}_{u}\right)-Y_{u}\right|\Big|\mathcal{F}_{\tau}\right].
\end{align*}

2) Let $\tau=\inf\left\{ u\in[s_{0},t_{0}]:Y_{u}\ge\lambda\right\} \wedge T$.
Then $\left\{ \sup_{t\in[s_{0},t_{0}]}Y_{u}\ge\lambda\right\} =\{\tau<T\}\cup\{\tau=t_{0},Y_{t_{0}}\ge\lambda\}\in\mathcal{F}_{\tau}$.
Therefore, by (\ref{Eqn-22}), 
\begin{align*}
\int_{\left\{ \sup_{u\in[s_{0},t_{0}]}Y_{u}\ge\lambda\right\} }Y_{\tau}d\mathbb{P} & =-\int_{\left\{ \sup_{u\in[s_{0},t_{0}]}Y_{u}\ge\lambda\right\} }\left(\mathbb{E}\left(Y_{t_{0}}\big|\mathcal{F}_{\tau}\right)-Y_{\tau}\right)d\mathbb{P}+\int_{\left\{ \sup_{u\in[s_{0},t_{0}]}Y_{u}\ge\lambda\right\} }Y_{t_{0}}d\mathbb{P}\\
 & \le\int_{\left\{ \sup_{u\in[s_{0},t_{0}]}Y_{u}\ge\lambda\right\} }\left[\sup_{u\in[s_{0},t_{0}]}\left|\mathbb{E}\left(Y_{t_{0}}\big|\mathcal{F}_{u}\right)-Y_{u}\right|+Y_{t_{0}}\right]d\mathbb{P}.
\end{align*}
\end{proof}
\begin{prop}
\label{Prop-1}Let $0\le s_{0}<t_{0}\le T$, and let $X^{\ast}=\sup_{u\in[s_{0},t_{0}]}\left|X_{u}\right|$.
Then for any $1<q\le p$,

\begin{equation}
\left\Vert X^{\ast}\right\Vert _{L^{q}}\le\frac{q}{q-1}\left[C_{p,h,\theta}^{1/p}A_{p,h}^{1/p}|t_{0}-s_{0}|^{h}+\left\Vert X_{t_{0}}\right\Vert _{L^{q}}\right].\label{Eqn-18}
\end{equation}
where $C_{p,h,\theta}$ is a constant which differs from the constant
$C_{p,h,\theta}$ in Lemma \ref{Lemma-1} by a multiple depending
only on $p$, and $\delta>0$ is an arbitrary constant.\end{prop}
\begin{proof}
Denote $Y=\sup_{u\in[s_{0},t_{0}]}\left|\mathbb{E}\left(X_{t_{0}}\big|\mathcal{F}_{u}\right)-X_{u}\right|+\left|X_{t_{0}}\right|$.
Then $\left\{ X^{\ast}\ge\lambda\right\} \subseteq\left\{ \sup_{u\in[s_{0},t_{0}]}X_{u}\ge\lambda\right\} $.
By Lemma \ref{Lemma-2}.2 and Lemma \ref{Lemma-1}
\begin{align*}
\left\Vert X^{\ast}\right\Vert _{L^{q}}^{q} & =q\int_{0}^{\infty}\lambda^{q-1}\mathbb{P}\left(X^{\ast}\ge\lambda\right)d\lambda\\
 & \le q\int_{0}^{\infty}\lambda^{q-2}\int_{\left\{ \sup_{u\in[s_{0},t_{0}]}X_{u}\ge\lambda\right\} }Yd\mathbb{P}d\lambda\\
 & \le q\int_{0}^{\infty}\lambda^{q-2}\int_{\left\{ X^{\ast}\ge\lambda\right\} }Yd\mathbb{P}d\lambda\\
 & =q\int_{\Omega}\left(\int_{0}^{X^{\ast}}\lambda^{q-2}d\lambda\right)Yd\mathbb{P}\\
 & =\frac{q}{q-1}\int_{\Omega}\left|X^{\ast}\right|^{q-1}Yd\mathbb{P}\\
 & \le\frac{q}{q-1}\left\Vert X^{\ast}\right\Vert _{L^{q}}^{q/q^{\prime}}\left\Vert Y\right\Vert _{L^{q}},
\end{align*}
where $q^{\prime}$ is the conjugate exponent of $q$. Therefore,
\begin{align*}
\left\Vert X^{\ast}\right\Vert _{L^{q}} & \le\frac{q}{q-1}\left\Vert Y\right\Vert _{L^{q}}\\
 & \le\frac{q}{q-1}\left(\left\Vert \sup_{u\in[s_{0},t_{0}]}\left|\mathbb{E}\left(X_{t_{0}}\big|\mathcal{F}_{u}\right)-X_{u}\right|\right\Vert _{L^{q}}+\left\Vert X_{t_{0}}\right\Vert _{L^{q}}\right)\\
 & \le\frac{q}{q-1}\left[C_{p,h,\theta}^{1/p}A_{p,h}^{1/p}|t_{0}-s_{0}|^{h}+\left\Vert X_{t_{0}}\right\Vert _{L^{q}}\right].
\end{align*}

\end{proof}

\section{Tail decay of the supremum}
\begin{defn}
The marginals of the process $X_{t}$ are said to have \emph{uniform
$\alpha$-exponential decay}, if there exist constants $\alpha>0$
, $C>0$ and $D>0$ such that
\begin{equation}
\mathbb{P}\left(|X_{t}|\ge\lambda\right)\le C\exp\left(-D\lambda^{\alpha}\right),\quad\text{for all }\lambda>0\;\text{and all }t\in[0,T].\label{Eqn-25}
\end{equation}

\end{defn}
We shall show that the distributions of the $\sup_{t\in[0,T]}|X_{t}|$
has $\alpha$-exponential decay, if and only if the marginals of $X_{t}$
have uniform $\alpha$-exponential decay. It follows from a simple
computation that
\begin{lem}
\label{Lemma}Let $X_{t}$, $t\in[0,T]$ be a random process satisfying
(\ref{Eqn-17}), and let $q>0$. Then
\[
\mathbb{E}\left(\left|X_{t}\right|^{q}\right)\le CD^{-q/a}\Gamma\left(\frac{q}{\alpha}+1\right).
\]

\end{lem}

\begin{thm}
\label{Thm}Let $X_{t}$, $t\in[0,T]$ be a random process satisfying
(\ref{Eqn-17}). Suppose that there exist constants $\alpha>0$ ,
$D>0$, and $\delta_{0}\ge0$ such that 
\[
\mathbb{P}\left(|X_{t}|\ge\lambda\right)\le C\exp\left(-D\lambda^{\alpha}\right),\quad\text{for all }\lambda>0\;\text{and all }t\in[0,T].
\]
Then
\[
\mathbb{P}\left(\sup_{t\in[0,T]}\left|X_{t}\right|\ge2\lambda\right)\le K\lambda^{-\frac{1}{h}}\exp\left[-\left(1-\frac{1}{ph}\right)D\lambda^{\alpha}\right],\quad\text{for all }\lambda\ge\delta_{0},
\]
where $K=4T\Big[C_{p,h,\theta}A_{p,h}\Big]^{\frac{1}{ph}}\left[1+C\left(\frac{p}{p-1}\right)^{p}\right]^{1-\frac{1}{ph}}$,
and the constant $C_{p,h,\theta}$ is the same as in Lemma \ref{Lemma-1}.\end{thm}
\begin{proof}
For $N\in\mathbb{N}_{+}$, let $I_{n}=[t_{n-1},t_{n}]=[(n-1)T/N,nT/N]$.
Then
\[
\left\{ \sup_{t\in[0,T]}\left|X_{t}\right|\ge2\lambda\right\} \subseteq\bigcup_{n=1}^{N}\left\{ \sup_{t\in I_{n}}\left|\mathbb{E}\left(X_{t_{n}}\big|\mathcal{F}_{t}\right)-X_{t}\right|\ge\lambda\right\} \bigcup\bigcup_{n=0}^{N-1}\left\{ \sup_{t\in I_{n}}\left|\mathbb{E}\left(X_{t_{n}}\big|\mathcal{F}_{t}\right)\right|\ge\lambda\right\} .
\]
Therefore,
\begin{align}
\mathbb{P}\left(\sup_{t\in[0,T]}\left|X_{t}\right|\ge2\lambda\right) & \le\sum_{n=1}^{N}\mathbb{P}\left(\sup_{t\in I_{n}}\left|\mathbb{E}\left(X_{t_{n}}\big|\mathcal{F}_{t}\right)-X_{t}\right|\ge\lambda\right)+\sum_{n=0}^{N-1}\mathbb{P}\left(\sup_{t\in I_{n}}\left|\mathbb{E}\left(X_{t_{n}}\big|\mathcal{F}_{t}\right)\right|\ge\lambda\right).\label{Eqn-28}
\end{align}
By Lemma \ref{Lemma-1},
\begin{equation}
\mathbb{P}\left(\sup_{t\in I_{n}}\left|\mathbb{E}\left(X_{t_{n}}\big|\mathcal{F}_{t}\right)-X_{t}\right|\ge\lambda\right)\le C_{p,h,\theta}A_{p,h}\frac{1}{\lambda^{p}}\left(\frac{T}{N}\right)^{ph}.\label{Eqn-27}
\end{equation}

We need to estimate $\mathbb{P}\left(\sup_{t\in I_{n}}\left|\mathbb{E}\left(X_{t_{n}}\big|\mathcal{F}_{t}\right)\right|\ge\lambda\right)$.
If $\alpha>p$, by Doob's inequality and Lemma \ref{Lemma}, 
\begin{align*}
\mathbb{E}\left(\sup_{t\in I_{n}}\left|\mathbb{E}\left(X_{t_{n}}\big|\mathcal{F}_{t}\right)\right|^{\alpha}\right) & \le\left(\frac{\alpha}{\alpha-1}\right)^{\alpha}\mathbb{E}\left(\left|X_{t_{n}}\right|^{\alpha}\right)\\
 & \le C\left(\frac{p}{p-1}\right)^{p}D^{-1}.
\end{align*}
If $\alpha\le p$, the above yields that 
\[
\mathbb{E}\left(\sup_{t\in I_{n}}\left|\mathbb{E}\left(X_{t_{n}}\big|\mathcal{F}_{t}\right)\right|^{\alpha}\right)\le\mathbb{E}\left(\sup_{t\in I_{n}}\left|\mathbb{E}\left(X_{t_{n}}\big|\mathcal{F}_{t}\right)\right|^{p}\right)^{\alpha/p}\le C\left(\frac{p}{p-1}\right)^{p}D^{-1}.
\]
Moreover, for any $q\ge2$, by a similar argument,
\begin{align*}
\mathbb{E}\left(\sup_{t\in I_{n}}\left|\mathbb{E}\left(X_{t_{n}}\big|\mathcal{F}_{t}\right)\right|^{\alpha q}\right) & \le\mathbb{E}\left(\sup_{t\in I_{n}}\left|\mathbb{E}\left(\left|X_{t_{n}}\right|^{\alpha}\big|\mathcal{F}_{t}\right)\right|^{q}\right)\\
 & \le\left(\frac{q}{q-1}\right)^{q}\mathbb{E}\left(\left|X_{t_{n}}\right|^{\alpha q}\right)\\
 & \le C\left(\frac{q}{D(q-1)}\right)^{q}\Gamma\left(q+1\right)\\
 & \le C(2D^{-1})^{q}\Gamma(q+1).
\end{align*}
Therefore,
\begin{align*}
\mathbb{E}\left[\exp\left(\frac{D}{4}\sup_{t\in I_{n}}\left|\mathbb{E}\left(X_{t_{n}}\big|\mathcal{F}_{t}\right)\right|^{\alpha}\right)\right] & =\sum_{q=0}^{\infty}\frac{(D/4)^{q}}{q!}\mathbb{E}\left(\sup_{t\in I_{n}}\left|\mathbb{E}\left(X_{t_{n}}\big|\mathcal{F}_{t}\right)\right|^{\alpha q}\right)\\
 & \le1+\frac{C}{4}\left(\frac{p}{p-1}\right)^{p}+C\sum_{q=2}^{\infty}2^{-q}\\
 & \le2\left[1+C\left(\frac{p}{p-1}\right)^{p}\right].
\end{align*}
By Chebyshev's inequality,
\begin{equation}
\mathbb{P}\left(\sup_{t\in I_{n}}\left|\mathbb{E}\left(X_{t_{n}}\big|\mathcal{F}_{t}\right)\right|\ge\lambda\right)\le2\left[1+C\left(\frac{p}{p-1}\right)^{p}\right]\exp\left(-\frac{D}{4}\lambda^{\alpha}\right).\label{Eqn-26}
\end{equation}

Therefore, by (\ref{Eqn-28}), (\ref{Eqn-27}) and (\ref{Eqn-26}),
\[
\mathbb{P}\left(\sup_{t\in[0,T]}\left|X_{t}\right|\ge2\lambda\right)\le C_{p,h,\theta}A_{p,h}\frac{N}{\lambda^{p}}\left(\frac{T}{N}\right)^{ph}+2N\left[1+C\left(\frac{p}{p-1}\right)^{p}\right]\exp\left(-\frac{D}{4}\lambda^{\alpha}\right).
\]
Setting $N$ to be the integer part of $\Big[C_{p,h,\theta}A_{p,h}T^{ph}\lambda^{-p}\exp\left(D\lambda^{\alpha}\right)\Big]^{\frac{1}{ph}}$
gives that
\[
\mathbb{P}\left(\sup_{t\in[0,T]}\left|X_{t}\right|\ge2\lambda\right)\le4T\Big[C_{p,h,\theta}A_{p,h}\Big]^{\frac{1}{ph}}\left[1+C\left(\frac{p}{p-1}\right)^{p}\right]^{1-\frac{1}{ph}}\lambda^{-\frac{1}{h}}\exp\left[-\left(1-\frac{1}{ph}\right)D\lambda^{\alpha}\right].
\]
\end{proof}
\begin{example}
We consider the tail decay of the supremum of a standard fractional
Brownian motion $B_{t}^{h}$, $t\in[0,T]$, with Hurst parameter $h\in(0,1)$,
that is, a Gaussian process with $B_{0}^{h}=0$ and covariance function
\[
R(t,s)=\frac{1}{2}\left(|t|^{2h}+|s|^{2h}-|t-s|^{2h}\right),\quad t,s\in[0,T].
\]
For the fractional Brownian motion, one has $B_{t}^{h}-B_{s}^{h}\sim N\left(0,\frac{1}{2}|t-s|^{2h}\right)$,
and therefore,
\[
\mathbb{E}\left(|B_{t}^{h}-B_{s}^{h}|^{p}\right)=A_{p}|t-s|^{ph},\quad t,s\in[0,T],
\]
where
\[
A_{p}=\dfrac{1}{\sqrt{\pi}}\Gamma\left(\dfrac{p+1}{2}\right).
\]
For any $t\in[0,1]$ and any $\lambda>0$, one has
\begin{align*}
\mathbb{P}\left(|B_{t}^{h}|\ge\lambda\right) & =\dfrac{1}{\sqrt{2\pi}}\int_{\sqrt{2}\lambda/t^{h}}^{\infty}\exp\left(-\dfrac{1}{2}u^{2}\right)du\\
 & \le\dfrac{t^{h}}{2\sqrt{\pi}\lambda}\exp\left(-\dfrac{\lambda^{2}}{t^{2h}}\right)\\
 & \le\dfrac{1}{2\sqrt{\pi}\lambda}\exp\left(-\lambda^{2}\right).
\end{align*}
Now put $\phi(\lambda)=\dfrac{1}{2\sqrt{\pi}\lambda}\exp\left(-\lambda^{2}\right)$.
For any $p>1/h$, by Lemma \ref{Thm}.2,
\[
\mathbb{P}\left(\sup_{t\in[0,1]}\left|B_{t}^{h}\right|\ge2\lambda\right)\le2\left[C_{p,h,\theta}A_{p}\right]^{\frac{1}{ph}}\lambda^{-\frac{1}{h}}\phi(\lambda)^{1-\frac{1}{ph}},\quad\text{for all }\lambda>0,
\]
where 
\[
C_{p,h,\theta}=[2\zeta(\theta)]^{p-1}\left(\frac{p}{p-1}\right)^{p}\left(\frac{4}{ph-1}\right)^{\theta(p-1)+1}\Gamma\left[\theta(p-1)+1\right]
\]
with an arbitrary constant $\theta>1$. Setting $p=2/h$, $\theta=p/(p-1)$
gives
\[
\mathbb{P}\left(\sup_{t\in[0,1]}\left|B_{t}^{h}\right|\ge2\lambda\right)\le\frac{C_{h}}{\lambda}\exp\left(-\frac{\lambda^{2}}{2}\right),\quad\text{for all }\lambda>0,
\]
where $C_{h}$ is a constant depending only on $h$. By scaling, we
deduce that
\[
\mathbb{P}\left(\sup_{t\in[0,T]}\left|B_{t}^{h}\right|\ge2\lambda\right)\le\frac{C_{h}T^{h}}{\lambda}\exp\left(-\frac{\lambda^{2}}{2T^{2h}}\right),\quad\text{for all }\lambda>0.
\]

\end{example}

\begin{example}
Let $a=(a_{1},a_{2},\cdots)$ be a sequence of real numbers, let $f_{k}(t)$,
$k=1,2,\cdots$, be a sequence of real functions defined on $[0,T]$,
and let $\xi_{k}$, $k=1,2,\cdots$, be i.i.d. Rademacher random variables.
In \cite{Pal30}, Theorem I, p. 339, R. Paley and A. Zygmund showed
that if
\[
\sum_{k=1}^{\infty}a_{k}^{2}<\infty
\]
and
\[
\int_{0}^{T}f_{k}(t)^{2}dt\le A,\quad k=1,2,\cdots,
\]
for some constant $A<\infty$,then for almost all $\omega\in\Omega$,
the series $\sum_{k=1}^{\infty}a_{k}\xi_{k}(\omega)f_{k}(t)$ converges
for a.e. $t\in[0,T]$, and the limit is an element in $L^{2}([0,T])$. 

Let $H\subseteq[0,T]\times\Omega$ be the set of $(t,\omega)$ at
which $\sum_{k=1}^{\infty}a_{k}\xi_{k}(\omega)f_{k}(t)$ converges.
Then the projection of $H$ on $\Omega$ has probability one. We shall
show that, under some stronger assumption, $\sum_{k=1}^{\infty}a_{k}\xi_{k}f_{k}(t)$
converges uniformly in $t\in H^{\omega}$ for any $\omega\in\Omega$.
Here $H^{\omega}=\{t\in[0,T]:(t,\omega)\in H\}$ has Lebesgue measure
$T$ for almost all $\omega\in\Omega$. Suppose that $h\in(0,1)$
and that $f_{k}(t)$, $k=1,2,\cdots$, be $h$-Hölder continuous with
Hölder constants $L_{k}$, that is,
\[
|f_{k}(t)-f_{k}(s)|\le L_{k}|t-s|^{h}.
\]
Suppose that
\[
\sum_{k=1}^{\infty}a_{k}^{2}\left[f_{k}(0)^{2}+L_{k}^{2}\right]<\infty.
\]
Then for a.e. $\omega\in\Omega$,
\[
\sum_{k=1}^{\infty}a_{k}\xi_{k}f_{k}(t)
\]
converges uniformly in $t\in[0,T]$.\end{example}
\begin{proof}
Let $X(t,\omega)=\sum_{k=1}^{\infty}a_{k}\xi_{k}(\omega)f_{k}(t)1_{H}(t,\omega)$.
And to simplify the notation, we refer to $\sum_{k=1}^{\infty}a_{k}\xi_{k}(\omega)f_{k}(t)1_{H}(t,\omega)$
by simply writing $\sum_{k=1}^{\infty}a_{k}\xi_{k}(\omega)f_{k}(t)$. 

For any $p>0$, by Khintchine's inequality,
\[
\mathbb{E}\left[\left|\sum_{k=1}^{\infty}a_{k}\xi_{k}f_{k}(0)\right|^{p}\right]\le C_{p}\left[\sum_{k=1}^{\infty}a_{k}^{2}f_{k}(0)^{2}\right]^{p/2},
\]
where
\[
C_{p}=\max\left(\sqrt{\frac{2^{p}}{\pi}}\Gamma\left(\frac{p+1}{2}\right),1\right).
\]
Similarly,
\[
\mathbb{E}\left[\left|\sum_{k=1}^{\infty}a_{k}\xi_{k}\left(f_{k}(t)-f_{k}(0)\right)\right|^{p}\right]\le C_{p}\left[\sum_{k=1}^{\infty}a_{k}^{2}L_{k}^{2}\right]^{p/2}t^{ph}.
\]

By the above, we see that
\begin{equation}
\mathbb{E}\left(|X_{t}|^{p}\right)\le2^{p/2}C_{p}\left[\sum_{k=1}^{\infty}a_{k}^{2}\left(f_{k}(0)^{2}+L_{k}^{2}\right)\right]^{p/2}\label{Eqn-3}
\end{equation}
and
\begin{equation}
\mathbb{E}\left(|X_{t}-X_{s}|^{p}\right)\le2^{p/2}C_{p}\left[\sum_{k=1}^{\infty}a_{k}^{2}L_{k}^{2}\right]^{p/2}|t-s|^{ph}\label{Eqn-4}
\end{equation}
for all $p>0$. Put 
\begin{equation}
A_{p,h}=\frac{2^{p}}{\sqrt{\pi}}\Gamma\left(\frac{p+1}{2}\right)\left[\sum_{k=1}^{\infty}a_{k}^{2}L_{k}^{2}\right]^{p/2}.\label{Eqn-5}
\end{equation}
Then condition (\ref{Eqn}) is satisfied for any $p>1/h$.

We now give an estimate of the tail decay of the marginals $X_{t}$.
For any $u\ge0$, by (\ref{Eqn-3}) and Stirling's formula,
\begin{align*}
\mathbb{E}\left[\exp\left(u|X_{t}|\right)\right] & \le\sum_{p=0}^{\infty}\frac{2^{p/2}u^{p}C_{p}}{p!}\left[\sum_{k=1}^{\infty}a_{k}^{2}\left(f_{k}(0)^{2}+L_{k}^{2}\right)\right]^{p/2}\\
 & \le K\sum_{p=0}^{\infty}2^{p}\frac{\Gamma\left(\frac{p+1}{2}\right)}{\Gamma\left(p+1\right)}\left[u^{2}\sum_{k=1}^{\infty}a_{k}^{2}\left(f_{k}(0)^{2}+L_{k}^{2}\right)\right]^{p/2}\\
 & =K\sum_{p=0}^{\infty}\frac{\sqrt{\pi}}{\Gamma\left(\frac{p}{2}+1\right)}\left[u^{2}\sum_{k=1}^{\infty}a_{k}^{2}\left(f_{k}(0)^{2}+L_{k}^{2}\right)\right]^{p/2}\\
 & \le K\left(\sum_{p=0}^{\infty}\frac{1}{\Gamma\left(\frac{p}{2}+1\right)^{2}}\left[u^{2}\sum_{k=1}^{\infty}a_{k}^{2}\left(f_{k}(0)^{2}+L_{k}^{2}\right)\right]^{p}\right)^{1/2}
\end{align*}
where $K$ is a universal constant that might be different from line
to line. Since
\[
\Gamma\left(\frac{p}{2}+1\right)^{2}\ge\Gamma\left(\frac{p+2}{2}\right)\Gamma\left(\frac{p+1}{2}\right)=2^{-p}\sqrt{\pi}\Gamma(p+1)=\frac{\sqrt{\pi}p!}{2^{p}},
\]
we obtain that
\[
\mathbb{E}\left[\exp\left(u|X_{t}|\right)\right]\le K\left(\sum_{p=0}^{\infty}\frac{1}{p!}\left[2u^{2}\sum_{k=1}^{\infty}a_{k}^{2}\left(f_{k}(0)^{2}+L_{k}^{2}\right)\right]^{p}\right)^{1/2}=K\exp\left[u^{2}\sum_{k=1}^{\infty}a_{k}^{2}\left(f_{k}(0)^{2}+L_{k}^{2}\right)\right].
\]
Now, by Chebyshev's inequality,
\[
\mathbb{P}\left(|X_{t}|\ge\lambda\right)\le K\exp\left[-u\lambda+u^{2}\sum_{k=1}^{\infty}a_{k}^{2}\left(f_{k}(0)^{2}+L_{k}^{2}\right)\right].
\]
Setting $u=\lambda\left[2\sum_{k=1}^{\infty}a_{k}^{2}\left(f_{k}(0)^{2}+L_{k}^{2}\right)\right]^{-1}$
gives that
\begin{equation}
\mathbb{P}\left(|X_{t}|\ge\lambda\right)\le K\exp\left[-\frac{\lambda^{2}}{2\sum_{k=1}^{\infty}a_{k}^{2}\left(f_{k}(0)^{2}+L_{k}^{2}\right)}\right].\label{Eqn-2}
\end{equation}
Now by Theorem \ref{Thm}, we have
\begin{equation}
\mathbb{P}\left(\sup_{t\in[0,T]}\left|\sum_{k=1}^{\infty}a_{k}\xi_{k}f_{k}(t)\right|\ge2\lambda\right)\le C_{h}\exp\left[-\frac{D_{h}\lambda^{2}}{\sum_{k=1}^{\infty}a_{k}^{2}\left(f_{k}(0)^{2}+L_{k}^{2}\right)}\right],\label{Eqn-6}
\end{equation}
where $C_{h}$, $D_{h}$ are constants depending only on $h\in(0,1)$.

Now, for any $u\in[0,1]$, denote $\sigma^{2}=\sum_{k=1}^{\infty}a_{k}^{2}\left(f_{k}(0)^{2}+L_{k}^{2}\right)$,
and let 
\begin{equation}
l(u)=\inf\left\{ l\in\mathbb{R}_{+}:\,\int_{0}^{l}\sum_{k=1}^{\infty}a_{k}^{2}\left(f_{k}(0)^{2}+L_{k}^{2}\right)1_{(k-1,k]}(s)ds>u\sigma^{2}\right\} .\label{Eqn-9}
\end{equation}
Then $l(u)\to\infty$ as $u\to1$. 

To show the a.s. uniform convergence of $\sum_{k=1}^{\infty}a_{k}\xi_{k}f_{k}(t)$,
we need to show that
\[
\mathbb{P}\left(\bigcap_{0<u<1}\left\{ \sup_{n\ge l(u)}\sup_{t\in[0,T]}\left|\sum_{k=n}^{\infty}a_{k}\xi_{k}f_{k}(t)\right|\ge2\lambda\right\} \right)=0
\]
for any $\lambda>0$. Clearly, it suffices to show that
\begin{equation}
\lim_{u\to1}\mathbb{P}\left(\sup_{n\ge l(u)}\sup_{t\in[0,T]}\left|\sum_{k=n}^{\infty}a_{k}\xi_{k}f_{k}(t)\right|\ge2\lambda\right)=0.\label{Eqn-7}
\end{equation}
Define
\[
Y_{u}=\sup_{t\in[0,T]}\left|\int_{0}^{l(u)}\sum_{k=1}^{\infty}a_{k}\xi_{k}f_{k}(t)1_{(k-1,k]}(s)ds\right|,\quad u\in[0,1].
\]
To prove(\ref{Eqn-7}), it suffices to show that
\begin{equation}
\lim_{u\to1}\mathbb{P}\left(\sup_{v\in[u,1]}Y_{v}\ge2\lambda\right)=0.\label{Eqn-8}
\end{equation}

We first note that, for $0\le u<v\le1$,
\[
|Y_{v}-Y_{u}|\le\sup_{t\in[0,T]}\left|\int_{l(u)}^{l(v)}\sum_{k=1}^{\infty}a_{k}\xi_{k}f_{k}(t)1_{(k-1,k]}(s)ds\right|.
\]
In fact, let $t^{\ast}\in\mathrm{arg}\max_{t\in[0,T]}\left|\int_{0}^{l(v)}\sum_{k=1}^{\infty}a_{k}\xi_{k}f_{k}(t)1_{(k-1,k]}(s)ds\right|$.
Then
\begin{align}
Y_{v}-Y_{u} & \le\left|\int_{0}^{l(v)}\sum_{k=1}^{\infty}a_{k}\xi_{k}f_{k}(t^{\ast})1_{(k-1,k]}(s)ds\right|-\left|\int_{0}^{l(u)}\sum_{k=1}^{\infty}a_{k}\xi_{k}f_{k}(t^{\ast})1_{(k-1,k]}(s)ds\right|\nonumber \\
 & \le\left|\int_{l(u)}^{l(v)}\sum_{k=1}^{\infty}a_{k}\xi_{k}f_{k}(t^{\ast})1_{(k-1,k]}(s)ds\right|\nonumber \\
 & \le\sup_{t\in[0,T]}\left|\int_{l(u)}^{l(v)}\sum_{k=1}^{\infty}a_{k}\xi_{k}f_{k}(t)1_{(k-1,k]}(s)ds\right|.\label{Eqn-14}
\end{align}
Similarly, $Y_{u}-Y_{v}\le\sup_{t\in[0,T]}\left|\int_{l(u)}^{l(v)}\sum_{k=1}^{\infty}a_{k}\xi_{k}f_{k}(t)1_{(k-1,k]}(s)ds\right|$.

For any $u<v$, by the definition of $l(u)$,  
\begin{equation}
\sum_{k=1}^{\infty}a_{k}^{2}\left(f_{k}(0)^{2}+L_{k}^{2}\right)\int_{l(u)}^{l(v)}1_{(k-1,k]}(s)ds=(v-u)\sigma^{2}.\label{Eqn-12}
\end{equation}
Now, applying (\ref{Eqn-6}) to the sequence $\left(a_{1}^{\prime},a_{2}^{\prime},\cdots\right)$
with 
\[
a_{k}^{\prime}=a_{k}\left(\int_{l(u)}^{l(v)}1_{(k-1,k]}(s)ds\right)^{1/2},\quad k\ge1,
\]
we obtain that
\begin{equation}
\mathbb{P}\left(|Y_{v}-Y_{u}|\ge2\lambda\right)\le C_{h}\exp\left[-\frac{D_{h}\lambda^{2}}{|v-u|\sigma^{2}}\right],\label{Eqn-15}
\end{equation}
where $C_{h}$ and $D_{h}$ are constants depending only on $h$ and
might vary from line to line. In particular, noting that $Y_{1}=0$,
\begin{equation}
\mathbb{P}\left(|Y_{v}|\ge2\lambda\right)\le C_{h}\exp\left[-\frac{D_{h}\lambda^{2}}{(1-u)\sigma^{2}}\right],\quad\text{for all }v\in[u,1].\label{Eqn-16}
\end{equation}
Since (\ref{Eqn-15}) implies that
\[
\mathbb{E}\left(|Y_{v}-Y_{u}|^{p}\right)\le C_{h}|v-u|^{p/2}
\]
for any $p>2$. We are now in a position to apply Theorem \ref{Thm}
again, and deduce that
\[
\mathbb{P}\left(\sup_{v\in[u,1]}Y_{v}\ge2\lambda\right)\le C_{h}\exp\left[-\frac{D_{h}\lambda^{2}}{(1-u)\sigma^{2}}\right].
\]
Thus, (\ref{Eqn-8}) follows readily.
\end{proof}

\section{An estimate for the up-crossing number of processes with increment
controls}

We now give an estimate for the up-crossing number of processes $X_{t}$
which satisfies the condition (\ref{Eqn}).
\begin{lem}
\label{Lemma-4}For any $0<q\le p$, $0<\alpha<\frac{h-1/p}{1/q-1/p}$
, and any random times $\tau$, $\sigma$ such that $0\le\sigma\le\tau\le T$,
it holds that
\begin{equation}
\mathbb{E}\left(|X_{\tau}-X_{\sigma}|^{q}\right)\le K_{q,\alpha}C_{p,h,\theta}^{q/p}A_{p,h}^{q/p}T^{qh}\mathbb{E}\left(\left|\dfrac{\tau-\sigma}{T}\right|^{\alpha}\right)^{1-q/p},\label{Eqn-19}
\end{equation}
where $K_{q,\alpha}=4^{q}\left[1-2^{-q(h-1/p)+(1-q/p)\alpha}\right]^{-1}$,
and the constant $C_{p,h,\theta}$ is the same as in Lemma \ref{Lemma-1}.\end{lem}
\begin{rem*}
It is noticed that, when $q=p$, $K_{q,\alpha}=4^{p}\left[1-2^{1-ph}\right]^{-1}$
for all $\alpha>0$.\end{rem*}
\begin{proof}
We first note that, by virtue of Lemma \ref{Lemma-1}.1 applied to
the filtration $\mathcal{F}_{t}=\mathcal{F}_{T}$, $0\le t\le T$,
one has 
\begin{equation}
\mathbb{E}\left(|X_{\tau}-X_{\sigma}|^{p}\right)\le2^{p-1}C_{p,h,\theta}A_{p,h}|t_{0}-s_{0}|^{ph}\label{Eqn-20}
\end{equation}
for any random times $\tau$ and $\sigma$ with $s_{0}\le\sigma\le\tau\le t_{0}$.

Now, for any random times $\tau$, $\sigma$ with $0\le\sigma\le\tau\le T$,
define
\[
A_{r,k}=\left\{ T(r-1)2^{-k}\le\sigma<Tr2^{-k}<T(r+1)2^{-k}\le\tau<T(r+2)2^{-k}\right\} ,\quad1\le r\le2^{k}-1,\,k\ge1.
\]
Then $\{\tau\not=\sigma\}=\bigcup_{r,k}A_{r,k}$, and the union of
sets is disjoint. Therefore,
\[
X_{\tau}-X_{\sigma}=\sum_{r,k}\left(X_{\tau}-X_{\sigma}\right)1_{A_{r,k}}.
\]
Let
\begin{gather*}
\tau_{r,k}=\left(\tau\vee\dfrac{Tr}{2^{k}}\right)\wedge\dfrac{T(r+1)}{2^{k}},\\
\sigma_{r,k}=\left(\sigma\vee\dfrac{Tr}{2^{k}}\right)\wedge\dfrac{T(r+1)}{2^{k}}.
\end{gather*}
Then
\[
X_{\tau}-X_{\sigma}=\sum_{r,k}\left(X_{\tau_{r+1,k}}-X_{\sigma_{r-1,k}}\right)1_{A_{r,k}}.
\]
Since $A_{r,k}$ are mutually disjoint, we have
\begin{align*}
\mathbb{E}\left(\left|X_{\tau}-X_{\sigma}\right|^{q}\right) & =\mathbb{E}\left(\left|\sum_{r,k}\left(X_{\tau_{r+1,k}}-X_{\sigma_{r-1,k}}\right)1_{A_{r,k}}\right|^{q}\right)\\
 & =\mathbb{E}\left(\sum_{r,k}\left|X_{\tau_{r+1,k}}-X_{\sigma_{r-1,k}}\right|^{q}1_{A_{r,k}}\right)\\
 & \le\sum_{r,k}\left[\mathbb{E}\left(\left|X_{\tau_{r+1,k}}-X_{\sigma_{r-1,k}}\right|^{p}\right)\right]^{q/p}\mathbb{P}(A_{r,k})^{1-q/p}.
\end{align*}
Note that $T(r-1)2^{-k}\le\sigma_{r-1,k}\le\tau_{r+1,k}<T(r+2)2^{-k}$.
By (\ref{Eqn-20}),
\begin{align}
\mathbb{E}\left(\left|X_{\tau}-X_{\sigma}\right|^{q}\right) & \le\sum_{r,k}\left(2^{p-1}C_{p,h,\theta}A_{p,h}\left(\dfrac{T}{2^{k-2}}\right)^{ph}\right)^{q/p}\mathbb{P}(A_{r,k})^{1-q/p}\nonumber \\
 & \le2^{q}C_{p,h,\theta}^{q/p}A_{p,h}^{q/p}\sum_{r,k}\left(\dfrac{T}{2^{k-2}}\right)^{qh}\mathbb{P}(A_{r,k})^{1-q/p}.\label{Eqn-10}
\end{align}
By the fact that $\bigcup_{r=1}^{2^{k}-2}A_{r,k}\subseteq\left\{ |\tau-\sigma|>T2^{-k}\right\} $
and Chebyshev's inequality,
\[
\sum_{r=1}^{2^{k}-1}\mathbb{P}(A_{r,k})\le\mathbb{P}\left(|\tau-\sigma|>T2^{-k}\right)\le2^{k\alpha}\mathbb{E}\left(\left|\dfrac{\tau-\sigma}{T}\right|^{\alpha}\right).
\]
By Jensen's inequality,
\begin{align}
\sum_{r=1}^{2^{k}-1}\mathbb{P}(A_{r,k})^{1-q/p} & \le2^{k}\left[2^{-k}\sum_{r=1}^{2^{k}-1}\mathbb{P}(A_{r,k})\right]^{1-q/p}\nonumber \\
 & \le2^{k[q/p+(1-q/p)\alpha]}\left[\mathbb{E}\left(\left|\dfrac{\tau-\sigma}{T}\right|^{\alpha}\right)\right]^{1-q/p}.\label{Eqn-11}
\end{align}
Therefore, by (\ref{Eqn-10}) and (\ref{Eqn-11}),
\begin{align*}
\mathbb{E}\left(\left|X_{\tau}-X_{\sigma}\right|^{q}\right) & \le2^{q}C_{p,h,\theta}^{q/p}A_{p,h}^{q/p}\sum_{k=1}^{\infty}\left(\dfrac{T}{2^{k-2}}\right)^{qh}\sum_{r=1}^{2^{k}-1}\mathbb{P}(A_{r,k})^{1-q/p}\\
 & \le4^{q}C_{p,h,\theta}^{q/p}A_{p,h}^{q/p}T^{qh}\sum_{k=1}^{\infty}2^{k[q/p+(1-q/p)\alpha-qh]}\left[\mathbb{E}\left(\left|\dfrac{\tau-\sigma}{T}\right|^{\alpha}\right)\right]^{1-q/p}.
\end{align*}
Since $\alpha<\frac{h-1/p}{1/q-1/p}$, setting $K_{q,\alpha}=4^{q}\sum_{k=2}^{\infty}2^{k[q/p+(1-q/p)\alpha-qh]}<4^{q}\left[1-2^{-q(h-1/p)+(1-q/p)\alpha}\right]^{-1}$
completes the proof.\end{proof}
\begin{defn}
\label{Def}$\{Y_{t}:t\in[0,T]\}$ be a random process. Let $D:0=t_{0}<\cdots<t_{N}=T$
be a finite subset of $[0,T]$. For any $a$, $b\in\mathbb{R}$, $a<b$,
let
\[
T_{0}=\inf\left\{ t\in D:Y_{t}<a\right\} ,T_{1}=\inf\left\{ t\in D:t>T_{0},Y_{t}>b\right\} ,
\]
\[
T_{2k}=\inf\left\{ t\in D:t>T_{2k-1},Y_{t}<a\right\} ,T_{2k+1}=\inf\left\{ t\in D:t>T_{2k},Y_{t}>b\right\} ,\quad k\ge1.
\]
The up-crossing number $U_{a}^{b}(D)$ of $Y_{t}$ through $[a,b]$
sampled in $D$ is given by 
\[
U_{a}^{b}(D)=\sup\left\{ k\ge1:T_{2k-1}\le T\right\} .
\]
And the up-crossing number $U_{a}^{b}$ of $Y_{t}$ through $[a,b]$
is defined as
\[
U_{a}^{b}=\sup_{D}U_{a}^{b}(D),
\]
where $\sup_{D}$ is taken over all finite subsets $D$ of any countable
dense subset of $[0,T]$.

By definition, one has
\[
\left\{ U_{a}^{b}(D)\ge k\right\} =\left\{ T_{2k-1}\le T\right\} ,\quad k\ge1
\]
and 
\[
\left\{ U_{a}^{b}(D)=k\right\} =\left\{ T_{2k-1}\le T,T_{2k+1}=\infty\right\} ,\quad k\ge1.
\]

\end{defn}
For the up-crossing number $U_{a}^{b}(D)$ of a general random process
$Y_{t}$, $t\in[0,T]$, we have the following elementary but useful
result.
\begin{lem}
\label{Lemma-3}With the same notation as in Definition \ref{Def},
one has
\[
(b-a)1_{\left\{ U_{a}^{b}(D)\ge k\right\} }\le-\left(Y_{T}-Y_{T_{2(k-1)}}\right)1_{\left\{ T_{2(k-1)}\le T,T_{2k-1}=\infty\right\} }+Y_{T_{2k-1}\wedge T}-Y_{T_{2(k-1)}\wedge T},
\]
for any $k\ge1$.
\end{lem}

\begin{prop}
\label{Prop}Let $D:0=t_{0}<\cdots<t_{N}=T$ be a finite subset of
$[0,T]$, and let $U_{a}^{b}(D)$ be the up-crossing number of $X_{t}$
through $[a,b]$ sampled in $D$. Then, for any $0<\delta<1-\frac{1}{ph}$,
\[
\mathbb{E}\left(U_{a}^{b}(D)^{\delta}\right)<\frac{K_{\delta}}{b-a}T^{h},
\]
where 
\[
K_{\delta}=2\left(C_{p,h,\theta}^{1/p}A_{p,h}^{1/p}+4^{q}\left[1-2^{-q(h-1/p)+(1-q/p)\alpha}\right]^{-1}\zeta\left(\frac{1-\delta}{1-\alpha(1-q/p)}\right)^{1-\alpha(1-q/p)}C_{p,h,\theta}^{q/p}A_{p,h}^{q/p}\right),
\]
and $q$ and $\alpha$ are any constants satisfying $\frac{\delta}{h-1/p}<q<1/h$
and $\frac{\delta}{1-q/p}<\alpha<\frac{h-1/p}{1/q-1/p}$.\end{prop}
\begin{proof}
By Lemma \ref{Lemma-3} and Lemma \ref{Lemma-4},
\begin{align*}
\frac{b-a}{k^{1-\delta}}\mathbb{P}\left(U_{a}^{b}(D)\ge k\right) & \le\frac{1}{k^{1-\delta}}\mathbb{E}\left[\sup_{u,v\in[0,T]}|X_{u}-X_{v}|1_{\left\{ T_{2(k-1)}\le T,T_{2k-1}=\infty\right\} }\right]+\frac{1}{k^{1-\delta}}\mathbb{E}\left(Y_{T_{2k-1}\wedge T}-Y_{T_{2(k-1)}\wedge T}\right)\\
 & \le\mathbb{E}\left[\sup_{u,v\in[0,T]}|X_{u}-X_{v}|1_{\left\{ T_{2(k-1)}\le T,T_{2k-1}=\infty\right\} }\right]+\frac{1}{k^{1-\delta}}\mathbb{E}\left(\left|Y_{T_{2k-1}\wedge T}-Y_{T_{2(k-1)}\wedge T}\right|\right)\\
 & \le\mathbb{E}\left[\sup_{u,v\in[0,T]}|X_{u}-X_{v}|1_{\left\{ T_{2(k-1)}\le T,T_{2k-1}=\infty\right\} }\right]\\
 & \quad+\frac{1}{k^{1-\delta}}K_{q,\alpha}C_{p,h,\theta}^{q/p}A_{p,h}^{q/p}T^{h}\mathbb{E}\left(\left|\dfrac{T_{2k-1}\wedge T-T_{2(k-1)}\wedge T}{T}\right|^{\alpha}\right)^{1-q/p},
\end{align*}
where $0<q\le p$, $0<\alpha<\frac{h-1/p}{1/q-1/p}$, and $K_{q,\alpha}=4^{q}\left[1-2^{-q(h-1/p)+(1-q/p)\alpha}\right]^{-1}$.
Since $\left\{ T_{2(k-1)}\le T,T_{2k-1}=\infty\right\} $, $k\ge1$,
are mutually disjoint, we have
\begin{align}
(b-a)\sum_{k=1}^{\infty}\frac{1}{k^{1-\delta}}\mathbb{P}\left(U_{a}^{b}(D)\ge k\right) & \le\mathbb{E}\left[\sup_{u,v\in[0,T]}|X_{u}-X_{v}|\right]\nonumber \\
 & \quad+K_{q,\alpha}C_{p,h,\theta}^{q/p}A_{p,h}^{q/p}T^{h}\sum_{k=1}^{\infty}\frac{1}{k^{1-\delta}}\mathbb{E}\left(\left|\dfrac{T_{2k-1}\wedge T-T_{2(k-1)}\wedge T}{T}\right|^{\alpha}\right)^{1-q/p}.\label{Eqn-13}
\end{align}
Since $0<\delta<1-\frac{1}{ph}$, one may choose $q<p$ and then $\alpha$
such that $\frac{\delta}{h-1/p}<q<1/h$ and $\frac{\delta}{1-q/p}<\alpha<\frac{h-1/p}{1/q-1/p}<1$.
By (\ref{Eqn-13}) and Hölder's inequality,
\begin{align}
(b-a)\sum_{k=1}^{\infty}\frac{1}{k^{1-\delta}}\mathbb{P}\left(U_{a}^{b}(D)\ge k\right) & \le\left[\mathbb{E}\left(\sup_{u,v\in[0,T]}|X_{u}-X_{v}|^{p}\right)\right]^{1/p}\nonumber \\
 & \quad+K_{q,\alpha}C_{p,h,\theta}^{q/p}A_{p,h}^{q/p}T^{h}\sum_{k=1}^{\infty}\frac{1}{k^{1-\delta}}\mathbb{E}\left(\left|\dfrac{T_{2k-1}\wedge T-T_{2(k-1)}\wedge T}{T}\right|\right)^{\alpha(1-q/p)}\nonumber \\
 & \le C_{p,h,\theta}^{1/p}A_{p,h}^{1/p}T^{h}+K_{q,\alpha}C_{p,h,\theta}^{q/p}A_{p,h}^{q/p}T^{h}\left(\sum_{k=1}^{\infty}k^{-\frac{1-\delta}{1-\alpha(1-q/p)}}\right)^{1-\alpha(1-q/p)}\nonumber \\
 & \quad\times\left[\sum_{k=1}^{\infty}\mathbb{E}\left(\left|\dfrac{T_{2k-1}\wedge T-T_{2(k-1)}\wedge T}{T}\right|\right)\right]^{\alpha(1-q/p)}.\label{Eqn-21}
\end{align}
It follows from $\alpha>\frac{\delta}{1-q/p}$ that $\frac{1-\delta}{1-\alpha(1-q/p)}>1$.
Therefore, $\zeta\left(\frac{1-\delta}{1-\alpha(1-q/p)}\right)=\sum_{k=1}^{\infty}k^{-\frac{1-\delta}{1-\alpha(1-q/p)}}<\infty$.
Note that the sequence $T_{j}\wedge T$, $j\ge0$, is increasing and
bounded by $T$. We deduce that
\begin{align*}
\sum_{k=1}^{\infty}\mathbb{E}\left(\left|\dfrac{T_{2k-1}\wedge T-T_{2(k-1)}\wedge T}{T}\right|\right) & \le\sum_{k=1}^{\infty}\mathbb{E}\left(\dfrac{T_{2k}\wedge T-T_{2(k-1)}\wedge T}{T}\right)\\
 & =\mathbb{E}\left(\lim_{k\to\infty}\dfrac{T_{2k}\wedge T-T_{0}\wedge T}{T}\right)\\
 & \le1.
\end{align*}
Now (\ref{Eqn-21}) and the above yield that
\begin{equation}
(b-a)\sum_{k=1}^{\infty}\frac{1}{k^{1-\delta}}\mathbb{P}\left(U_{a}^{b}(D)\ge k\right)\le C_{p,h,\theta}^{1/p}A_{p,h}^{1/p}T^{h}+K_{q,\alpha}\zeta\left(\frac{1-\delta}{1-\alpha(1-q/p)}\right)^{1-\alpha(1-q/p)}C_{p,h,\theta}^{q/p}A_{p,h}^{q/p}T^{h}.\label{Eqn-23}
\end{equation}
Since 
\[
(k+1)^{\delta}-k^{\delta}=k^{\delta}\left[\left(1+\frac{1}{k}\right)^{\delta}-1\right]<k^{\delta}\cdot\frac{1}{k}=\frac{1}{k^{1-\delta}}\le\frac{2}{(k+1)^{1-\delta}},
\]
for any $0<\delta<1$, summing by parts, we deduce from (\ref{Eqn-23})
that 
\begin{align*}
\mathbb{E}\left(U_{a}^{b}(D)^{\delta}\right) & =\sum_{k=1}^{\infty}k^{\delta}\mathbb{P}\left(U_{a}^{b}(D)=k\right)\\
 & =\sum_{k=1}^{\infty}\left(k^{\delta}-(k-1)^{\delta}\right)\mathbb{P}\left(U_{a}^{b}(D)\ge k\right)\\
 & \le2\sum_{k=1}^{\infty}\frac{1}{k^{1-\delta}}\mathbb{P}\left(U_{a}^{b}(D)\ge k\right)\\
 & \le\frac{2}{b-a}\left(C_{p,h,\theta}^{1/p}A_{p,h}^{1/p}+K_{q,\alpha}\zeta\left(\frac{1-\delta}{1-\alpha(1-q/p)}\right)^{1-\alpha(1-q/p)}C_{p,h,\theta}^{q/p}A_{p,h}^{q/p}\right)T^{h}.
\end{align*}
This completes the proof.

By Fatou's lemma and Proposition \ref{Prop}, one has the following\end{proof}
\begin{thm}
Let $U_{a}^{b}$ be the up-crossing number of $X_{t}$ through $[a,b]$.
Then, for any $0<\delta<1-\frac{1}{ph}$,
\[
\mathbb{E}\left[\left(U_{a}^{b}\right)^{\delta}\right]<\frac{K_{\delta}}{b-a}T^{h},
\]
where the constant $K_{\delta}$ is the same as in Proposition \ref{Prop}.
In particular, $U_{a}^{b}<\infty$ a.s.
\end{thm}

\end{document}